\documentclass[12pt]{amsart}

\usepackage{graphicx} 

\usepackage{color}
\usepackage{amsmath}
\usepackage{amssymb}
\usepackage{amsbsy}
\usepackage{amsthm}

\usepackage{amsfonts}
\usepackage{tikz}

\newtheorem{lemma}{Lemma}[section]%
\newtheorem{theorem}[lemma]{Theorem}%
\newtheorem{definition}[lemma]{Definition}%
\newtheorem{hypothesis}[lemma]{Hypotheses}%
\newtheorem{notation}[lemma]{Notation}%
\newtheorem{problem}{Problem}%
\newtheorem{question}[problem]{Question}%
\newtheorem{proposition}[lemma]{Proposition}%
\newtheorem{remark}[lemma]{Remark}%

\title{Structure of Cayley Codes}
\author{Vishnuram Arumugam, Cheryl E. Praeger, Daniel Rademacher}
\address{The University of Western Australia, Perth WA 6009 {\text{\rm(Arumugam and  Praeger)}}; RWTH Aachen University, Aachen 52062, Germany  {\text{\rm(Rademacher)}}}
\email{vishnuram.arumugam@research.uwa.edu.au {\text{\rm(Arumugam)}};\newline  cheryl.praeger@uwa.edu.au {\text{\rm(Praeger)}};   daniel.rademacher@rwth-aachen.de {\text{\rm(Rademacher)}}}
\date{June 2026}

\newcommand{\Cay}{\operatorname{Cay}}

\def\F{\mathbb{F}}
\def\FF{\mathcal{F}}
\def\T{\mathcal{T}}
\def\Sym{{\rm Sym}}
\def\Cay{{\rm Cay}}

\def\Aut{{\rm Aut}}

\def\la{\langle}
\def\ra{\rangle}
\def\C{\mathbf{C}}

\def\hatpsi{\widehat{\psi}}
\def\Ecirc{E^\circ}

\begin{document}

\begin{abstract}
	\noindent
	
	Cayley codes, introduced by Kaufman and Wigderson, are linear codes constructed from a Cayley graph and a smaller linear code. We explore general properties of the class of Cayley codes for finite groups. In particular we give a reduction to Cayley codes for connected Cayley graphs that maintains code properties such as rate, minimum distance and symmetry. Also, for  a given Cayley code, we identify a family of symmetric Cayley codes, each associated with a normal edge-transitive Cayley graph, such that the given Cayley code embeds into the direct sum of the symmetric Cayley codes. We analyse several families of examples, in particular studying the behaviour of the Cayley code construction under forming direct products and cartesian products of Cayley graphs, and we pose a number of open questions.
	
	
	
	\medskip\noindent
	{\bf {Corresponding author}:}\quad Cheryl E Praeger
	
\end{abstract}

\maketitle

\section{Introduction}

In 2010, Kaufman and Wigderson introduced a new construction of linear codes, which they called Cayley codes, taking as input a Cayley graph and a smaller linear code  (see~\cite{KW2010}, or  \cite[Definition 19]{KW2016} for the journal version). They proved~\cite[Theorem 7]{KW2016} that, for suitable input graphs and codes, the Cayley code construction preserves certain bounds on the rate or relative distance, and  desirable code properties such as being symmetric. They presented explicit families of Cayley codes based on Cayley graphs of elementary abelian $2$-groups~\cite[Section 4]{KW2016}, obtaining codes of length $n$ which were `almost good' in the sense that the relative distance was $O((\log\log n)^{-2})$, rather than being strictly bounded away from zero, see ~\cite[Theorem 11]{KW2016}. Kaufman and Lubotzky \cite{KL2012} commented in 2012 that `\emph{it seems that all the known explicit constructions of symmetric codes give codes which are not good}', and provided, via the Cayley code construction, the first explicit infinite family of good binary symmetric codes based on the edge-transitive bipartite Ramanujan graphs -- particular Cayley graphs for the projective linear groups ${\rm PGL}_2(q)$,~\cite{LPS88, LSV05}  (see Notation~\ref{not:def} for definitions of  technical terms such as  `good'). Becker~\cite{B2016}  explored the impact of the expansion and symmetry properties of Ramanujan graphs on the properties of these Cayley codes; while Kaufman and Oppenheim \cite{MR4602836} developed a higher dimensional version.   

The aim of this paper is a little different. We explore general properties of the class of Cayley codes for finite groups. We show how to reduce to Cayley codes for connected Cayley graphs while maintaining the same rate and minimum distance and symmetry propertties (Theorem~\ref{thm:disccaycode}). We identify a family of symmetric Cayley codes related to an arbitrary given  Cayley code (Theorem~\ref{t:SiBi}). Normal edge-transitive Cayley graphs introduced in \cite{P99} play a critical role in the second reduction. Finally we give several families of examples, in particular exploring  the behaviour of the Cayley code construction under the Cartesian product construction for graphs (Theorem \ref{t:cprod}). 

We give broad statements of the major results in Subsection~\ref{s:results}, and to facilitate these statements we introduce the relevant concepts in Notation~\ref{not:def}. The final Subsection~\ref{s:future} briefly discusses some open questions about the structure of finite Cayley codes.

\begin{notation}\label{not:def}{\rm 
		A \emph{linear $[n,k,d]$-code} over a finite field $\F$ is a $k$-dimens\-ional subspace $C$ of $\F^X$, where $|X|=n$,  such that the minimum (Hamming) weight of the non-zero \emph{codewords} (elements of $C$) is $d$. A  family $\FF$ of such codes with length $n\to\infty$ is called \emph{good} if there exists $\epsilon>0$ such that the \emph{rate} $r(C):=\frac{k}{n}$ and the \emph{relative distance} $\delta(C):=\frac{d}{n}$ are both at least $\epsilon$. The \emph{automorphism group} $\Aut(C)$ of a linear code $C<\F^X$ is the subgroup of the symmetric group $\Sym(X)$ that leaves $C$ invariant in its induced action on $\F^X$, and $C$ is called \emph{symmetric} if $\Aut(C)$ is  transitive on $X$. We regard $\F^X$ as the $\F$-vector-space of functions $f:X\to \F$, so that in particular, for a subset $Y\subseteq X$, the restriction map $f\to f|_Y$ defines a natural projection $\F^X\to \F^Y$.
		
		For a group $G$ and an inverse-closed subset $S\subseteq G\setminus\{1\}$, the \emph{Cayley graph} $\Gamma=\Cay(G,S)=(G,E)$ 
		is the graph with vertex-set $V\Gamma=G$, and edge-set $E\Gamma=E:=\{\{g,sg\}\mid g\in G, s\in S\}$. For $g\in G$, the edge-subset 
		\begin{equation}\label{e:eg}
			E_g:=\{ \{g,sg\}\mid s\in S\}\quad \text{is the set of edges incident with $g$,}
		\end{equation}
		the map 
		\begin{equation}\label{e:phig}
			\phi_g:S\to E_g \quad\text{given by}\quad \phi_g:s\to \{g,sg\}\quad \text{for $s\in S$}
		\end{equation}
		is a bijection, and we obtain a linear projection $\phi: \F^E\to \F^S$, where $\phi:f\to \phi_g\circ f$.
		The \emph{Cayley code} $\C(G,S,B)$ corresponding to a given Cayley graph $\Cay(G,S)$ and linear code $B\leq \F^S$, is 
		\begin{equation}\label{e:cc}
			\C(G,S,B):=\{ f\in\F^E \mid \forall\ g\in G,\ \phi_g\circ f\in B \}.
		\end{equation}
		That $\C(G,S,B)$ is a linear code in $\F^E$ follows from the fact that $B$ is a linear code in $\F^S$. The automorphism group $\Aut(\C(G,S,B))$ contains the group $G\rtimes A(G,S,B)$, where 
		\begin{equation} \label{e:Autcc}
			A(G,S,B)=
			\{ \sigma\in\Aut(G)\mid S^\sigma=S, \sigma^S\in\Aut(B)\} 
		\end{equation}
		and in particular, if $G\rtimes A(G,S,B)$ is transitive on the edge-set $E$ then   $\Cay(G,S)$ is normal edge-transitive and also $\C(G,S,B)$ is symmetric (Proposition~\ref{p:cayc-props} - this strengthens a result \cite{KL2012, KW2010, KW2016}, see Remark~\ref{rem:strong}).
	}
\end{notation}

\subsection{The main results}\label{s:results}

A Cayley graph $\Cay(G,S)$ is disconnected if and only if  the subgroup $H:=\langle S \rangle$ is a proper subgroup of $G$ and the connected component of $\Cay(G,S)$ containing the identity is $\Cay(H,S)$ (Lemma~\ref{lem:cay1}). Thus we can form both Cayley codes $\C(G,S,B)$ and $\C(H,S,B)$, for appropriate $B$, and our first result shows how they are related. 

\begin{theorem}\label{thm:disconn}
	Let $G$ be a finite group and $S$ a non-empty inverse-closed subset of $G\setminus\{1\}$ such that $H:=\la S\ra \ne G$. Then for each linear code $B\leq \F^S$, 
	\begin{enumerate}
		\item[(a)]  the Cayley code $\C(G,S,B)$ decomposes as a direct sum of $|G:H|$ linear codes, each isomorphic to   $\C(H,S,B)$.
		
		\item[(b)] Moreover,  $\C(H,S,B)$ and $\C(G,S,B)$ have the same rate and minimum distance,  and  $\C(H,S,B)$ is symmetric if and only if $\C(G,S,B)$ is symmetric.
	\end{enumerate}
\end{theorem}
This result allows us to  restrict to  connected Cayley graphs if we wish. It follows immediately from a more technical version Theorem~\ref{thm:disccaycode} which is proved in Section~\ref{ss:disc}.

We next show that each Cayley code determines a family of symmetric Cayley codes. For $G, S, B, A(G,S,B)$ as in Notation~\ref{not:def}, 
the group $A(G,S,B)$  leaves invariant a unique finest partition  $S=\cup_{i\in\mathbf{I}} S_i$ such that each $S_i$ is inverse-closed and $A(G,S,B)$-invariant, and we have $B\leq \F^S=\oplus_{i\in\mathbf{I}}\F^{S_i}$. For each $i$, let $B_i$ be the restriction  of $B$ to $\F^{S_i}$. We show that each of the corresponding Cayley codes $\C(G,S_i,B_i)$ is symmetric.

\begin{theorem}\label{t:SiBi}
	With the notation above, for each $i\in\mathbf{I}$, the Cayley graph $\Cay(G,S_i)$ is normal edge-transitive and the Cayley code $\C(G,S_i,B_i)$ is symmetric. Further, if $B=\oplus_{i\in\mathbf{I}}B_i$, then 
	$\C(G,S,B)= \oplus_{i\in\mathbf{I}} \C(G,S_i,B_i)$.
\end{theorem}

This result is proved in Section~\ref{s:decomp}. Finally we look at various examples of Cayley codes. In Section~\ref{s:ex} we consider  Cayley codes  for several standard choices of the input code $B$, and for the family of Cayley graphs which are cycles. In the final Section~\ref{s:prod}  we show how the Cayley code construction behaves with respect to forming direct products and Cartesian products of Cayley graphs (see Definition~\ref{def:prod}). The Cartesian product of Cayley graphs $\Cay(G,S)$ and $\Cay(H,T)$ is $\Cay(G \times H,S \dot\cup T)$ (where we identify $S, T$ with the subsets  $S\times 1$, $1\times T$ of $G\times H$), and the Cayley code for the Cartesian product is described in Theorem~\ref{t:cprod} in terms of the Cayley codes for the Cartesian factors. The length, rate and distance of this Cayley code are described in terms of the parameters for the Cartesian factors in Remark~\ref{r:cprod}. The direct product of $\Cay(G,S)$ and $\Cay(H,T)$ is $\Cay(G \times H, S \times T)$, and our result for Cayley codes for the corresponding Cayley code is not so precise: we show (Theorem~\ref{t:cprod}) that a subcode projects onto a tensor product of the Cayley codes for the factors. Theorem~\ref{t:cprod} follows from Propositions~\ref{p:cartesianprod} and~\ref{p:dprod}.

\begin{theorem}\label{t:cprod}
	Let $G, H$ be finite groups  with inverse-closed generating sets $S, T$, respectively, let $\F$ be a finite field, and let $A\leq \F^S, B\leq \F^T$ be linear codes. Then 
	\begin{center}
		$\C(G \times H,S \dot\cup T, A \oplus B) = (\C(G,S,A) \otimes \mathbb{F}^{|H|}) \oplus (\mathbb{F}^{|G|} \otimes \C(H,T,B))$,
	\end{center}
	and a certain sub-code of $\C(G \times H,S \times T, A \otimes B)$ projects onto 
	\[ 
	\C(G,S,A) \otimes \C(H,T,B).
	\]
\end{theorem}

The paper is structured as follows.
Section~\ref{sec:CC} contains basic preliminary concepts for graphs and codes with the main focus on Cayley graphs.
In Section~\ref{s:cayc} we give a careful discussion of the Cayley code construction, and in the final  
Sections~\ref{ss:disc}--\ref{s:prod} we prove the main results as noted above.

\section{Spaces, graphs and codes}\label{sec:CC}

In this section we introduce the notation and basic properties of the codes and graphs we will work with, especially Cayley graphs.

\subsection{Vector spaces}\label{s:vsp}

As mentioned in the introduction, a linear code over a finite field $\F$ is a subspace of a finite vector space $\F^X$ of dimension $n=|X|$. To avoid confusion later, it is  helpful to view $\F^X$ as the set of all functions $f:X\to \F$ with pointwise addition and multiplication, that is,  
\[
af+a'f':x\to a(xf) + a'(xf')\quad \text{for $f,f'\in \F^X, a,a'\in \F$.}
\]
For a subset $Y\subseteq X$ and $f\in\F^X$, we denote by $f|_Y$ the restriction of $f$ to $Y$, so $f|_Y\in\F^Y$. Each expression of $X$ as a disjoint union $X=\bigcup_{k\in K}X_k$, for some index set $K$, corresponds to a direct sum decomposition $\F^X = \bigoplus_{k\in K} \F^{X_k}$ such that $f\in\F^X$ corresponds to the tuple $f=(f_k)_{k\in K}$ with $f_k = f|_{X_k}$ for $k\in K$.  If  $\rho\in\Sym(X)$ preserves the decomposition $X=\bigcup_{k\in K}X_k$, that is to say, if $\rho$ permutes the $X_k$ among themselves, then $\rho$ induces a permutation of $K$ and we may write $(X_k)^\rho = X_{k\rho}$ for each $k\in K$. Then for $f=(f_k)_{k\in K}$,  the restriction $(\rho\circ f )|_{X_k}$ first maps $X_k\to X_{k\rho}$ under $\rho$ and then $f$ induces the restriction $f|_{X_{k\rho}}$ on $X_{k\rho}$. Thus 
\begin{equation}\label{e:action}
	(\rho\circ f)_k = (\rho\circ f )|_{X_k} = \rho\circ (f|_{X_{k\rho}}) = \rho\circ f_{k\rho}, \quad \mbox{so $\rho\circ f=(\rho\circ f_{k\rho})_{k\in K}$.}
\end{equation}
We summarise this discussion in the next lemma.

\begin{lemma}\label{lem:Vdecomp}
	For a field $\F$, a set $X$, and a disjoint union $X=\bigcup_{k\in K}X_k$, for some index set $K$, the following hold. 
	\begin{enumerate}
		\item[(a)] $\F^X = \bigoplus_{k\in K} \F^{X_k}$ such that each $f=(f_k)_{k\in K}\in\F^X$ with $f_k = f|_{X_k}$, the restriction to $X_k$, for $k\in K$;
		\item[(b)] if $\rho\in\Sym(X)$ preserves the decomposition $X=\bigcup_{k\in K}X_k$, then for $f=(f_k)_{k\in K}\in\F^X$ we have $\rho\circ f=(\rho\circ f_{k\rho})_{k\in K}$. 
	\end{enumerate}
\end{lemma}

\subsection{Linear codes and their parameters}\label{s:linc}

As introduced in Notation~\ref{not:def}, a \emph{linear $[n,k,d]$-code} over $\F$ is a $k$-subspace $C$ of $\F^X$, where $|X|=n$, and we summarise in Table~\ref{tab1} some relevant parameters for such codes. Those permutations $\rho\in\Sym(X)$ which leave $C$ invariant (setwise) in their induced action on $\F^X$ are called \emph{automorphisms} of $C$ and form the automorphism group $\Aut(C)$ of the code $C$.

\begin{table}[ht]
	\caption{Some parameters for linear $[n,k,d]$-codes $C$ in $\F^X$, where $|X|=n$}\label{tab1}
		\begin{tabular}{lll}
			\hline
			Parameter name & Notation & Comments\\
			\hline
			Length                   &$n$&$n=|X|$\\
			Rank            &$k$ & $k=\dim(C)$ elements of $C$ are  \\
			&& called \emph{codewords} \\
			Rate &$r$            & $r = r(C)=k/n$\\
			Distance&$d$ & minimum \emph{weight} (number of non-  \\
			&&zero entries) of a nonzero codeword \\
			Relative distance &$\delta$ & $\delta = \delta(C)=d/n$ sometimes called the  \\
			&& normalised distance \\
			Automorphism group & $\Aut(C)$ & subgroup of $\Sym(X)$ leaving $C$ \\
			&&invariant; $C$ is \emph{symmetric} if $\Aut(C)$ \\
			&& is transitive on $X$\\
			\hline
		\end{tabular}
	\end{table}
	
	\subsection{Cayley graphs}\label{s:caygraphs}
	
	Recall from Notation~\ref{not:def} the notion of a Cayley graph $\Gamma = \Cay(G,S)$ for 
	a finite group $G$ and non-empty inverse-closed subset $S\subseteq G\setminus\{1\}$, that is, $S^{-1}=\{ s^{-1}\mid s\in S\} = S$. Each $g\in G$ is incident precisely with the edges in the set $E_g$ defined in \eqref{e:eg}, so $g$ is  incident with $|E_g|=|S|$ edges, and hence $\Gamma$ is \emph{regular} of valency $|S|$.
	The property $S^{-1}=S$ ensures that $\Gamma$ is an \emph{undirected graph} (that is, $g$ is adjacent to $g'$ if and only if $g'$ is adjacent to $g$); and $1\not\in S$ ensures that $\Gamma$ has no loops.

	It follows from the definition of  $E\Gamma$ that the connected component of $\Gamma$ containing the identity element $1$ has vertex set $H=\la S\ra$ and edge set
	\begin{equation}\label{e:EH}
		\Ecirc:=\{\{h,sh\}\mid h\in H, s\in S\} = \cup_{h\in  H} E_1 h,\ \mbox{with $E_1$ as in \eqref{e:eg}}
	\end{equation} 
	and hence is equal to  $\Cay(H,S)=(H,\Ecirc )$.
	For an arbitrary element $t\in G$, the connected component of $\Gamma$ containing the $t$ has vertex set equal to the coset $Ht=\la S\ra t$, and edge set 
	\begin{equation}\label{e:Et}
		\{\{ht, sht\}\mid h\in H, s\in S\} = \Ecirc t
	\end{equation}
	and so is the graph $(Ht, \Ecirc t)$. The  action of $t$ by right multiplication on $G$ and on pairs($t:\{h,sh\}\to \{ht, sht\}$) induces an isomorphism from $\Cay(H,S)=(H,\Ecirc )$ to  $(Ht, \Ecirc  t)$. We summarise this discussion in the following statement.
	
	\begin{lemma}\label{lem:cay1}
		Let $G$ be a finite group and $S$ a non-empty inverse-closed subset of $G\setminus\{1\}$. Let $H:=\la S\ra$ have index $u$ in $G$, and let $\T=\{t_1,\dots,t_u\}$ be a set of coset representatives for $H$ in $G$ with $t_1=1$. Then 
		\begin{enumerate}
			\item[(a)] the Cayley graph $\Cay(G,S)$ has exactly $u$ connected components, namely $(Ht_i, \Ecirc  t_i)$ (the connected component containing $t_i$) for $i=1,\dots, u$. 
			\item[(b)] Moreover $(H, \Ecirc )=\Cay(H,S)$, and for each $i$, the map 
			\[
			\text{$t_i:h\to ht_i$, and $\{h,sh\}\to \{ht_i,sht_i\}$ (for $h\in H, s\in S$)}
			\]
			defines a graph  isomorphism from $\Cay(H,S)$ to $(Ht_i, \Ecirc  t_i)$. 
			\item[(c)] $\Cay(G,S)$ is connected if and only if $S$ generates $G$.
		\end{enumerate} 
	\end{lemma}
	
	Theorem~\ref{thm:disccaycode}, proved in Section~\ref{ss:disc}, shows how Cayley codes for disconnected graphs are related to the Cayley codes corresponding to the connected components. As a result we often assume that $\la S\ra  = G$ and $\Gamma$ is connected.
	
	\subsection{Automorphisms of Cayley graphs} \label{s:aut}
	
	Each Cayley graph $\Gamma = \Cay(G,S)$ admits as a subgroup of automorphisms the group 
	\begin{equation} \label{e:Aut}
		\mbox{$A(\Gamma) := G\rtimes A(G,S)$, where $A(G,S)=
			\{ \sigma\in\Aut(G)\mid S^\sigma=S\}$}, 
	\end{equation}
	where for all $g,x\in G$ and $\sigma\in A(G,S)$, 
	$
	x:g\to gx\quad \text{and}\quad \sigma:g\to g^\sigma$.
	It is straightforward to check that each of these maps preserves $E\Gamma$, and hence defines a graph automorphism, so $A(\Gamma)\leq \Aut(\Gamma)$. In particular, the group $G$ acts by right multiplication as a \emph{regular permutation group} on $V\Gamma$ (transitive with trivial vertex-stabilisers), so $\Gamma$ is always vertex-transitive (with apologies for the dual use of the term `regular' from graph theory and permutation group theory).

	The subgroup $A(\Gamma)$ is the normaliser of $G$ in $\Aut(\Gamma)$; it is the largest subgroup of $\Aut(\Gamma)$ which preserves the structure of $\Gamma$ as a Cayley graph of $G$. Although $A(\Gamma)$ is vertex-transitive, it is not always edge-transitive on $\Gamma$. In fact  $A(\Gamma)$ is transitive on $E\Gamma$ if and only if either $A(G,S)$ is transitive on $S$, or $S$ is the disjoint union $S_0\, \dot\cup\, S_0^{-1}$ such that $S_0$ (and hence also $S_0^{-1}$) is an $A(G,S)$-orbit (see \cite[Proposition 1(b)]{P99}). Cayley graphs with this property were called \emph{normal edge-transitive} in \cite{P99}. Thus $\Gamma$ is normal edge-transitive if and only if $A(\Gamma)$ is edge-transitive. Note that, if $S$ generates $G$, then $A(G,S)$ acts faithfully on $S$, and so can be identified with a subgroup of $\Sym(S)$.
	
	In our application to Cayley codes we will need to work with some subgroup of $A(G,S)$ and we now set up the general context for this. We use the following notation: for each $g\in G$ and subgroup $L\leq A(G,S)$ we denote by $g^L=\{ g^\sigma\mid \sigma\in L\}$ the $L$-orbit in $G$ containing $g$.
	
	\begin{notation}\label{not:L}
		{\rm 
			Let $G$ be a finite group, $S$ an inverse-closed subset of $G\setminus\{1\}$, and let $1\leq L\leq A(G,S)$ with $A(G,S)$ as in \eqref{e:Aut}. Let  $S=\cup_{i\in\mathbf{I}}S_i$ be the partition of $S$ such that, for each $i$, and $s\in S_i$, we have $S_i= s^L\cup(s^{-1})^L$. That is to say, the partition $\cup_{i\in\mathbf{I}}S_i$ is the (unique) finest partition of $S$ such that each part $S_i$ is both $L$-invariant and inverse-closed.
		}
	\end{notation}
	
	For example, if the group $L=1$ in Notation~\ref{not:L}, then each  $S_i$ has the form $\{s,s^{-1}\}$ with size $1$ or $2$, and each connected component of the Cayley subgraph $\Cay(G, \{s,s^{-1}\})$ is a cycle of length $|s|$. Similarly, as noted above, if $L$ is such that $G\rtimes L$ is edge-transitive on $\Cay(G,S)$, then the partition consists of a single part and in particular $\Cay(G,S)$ is normal edge-transitive.

	The partition of $S$ in Notation~\ref{not:L} corresponds to an edge-disjoint decomposition of the Cayley graph $\Cay(G,S)$ into  normal edge-transitive Cayley graphs for $G$.

	\begin{lemma}\label{l:L}
		Let $G, S, L$ and the partition $S=\cup_{i\in\mathbf{I}}S_i$ be as in Notation~$\ref{not:L}$, and let $\Gamma=\Cay(G,S) = (G, E)$. Then the following hold.
		\begin{enumerate}
			\item[(a)] $\Gamma$ is an edge-disjoint union of Cayley subgraphs $\Gamma_i=\Cay(G,S_i)$, for $i\in\mathbf{I}$, and the group $G\rtimes L \leq \cap_{i\in \mathbf{I}}\, \Aut(\Gamma_i)$;
			
			\item[(b)] for each $i\in\mathbf{I}$, the edge-set $E\Gamma_i=\{ \{g,sg\} \mid g\in G, s\in S_i\}$ and is a  $(G\rtimes L)$-orbit, so $\Gamma_i$ is a normal edge-transitive Cayley graph.
		\end{enumerate}
	\end{lemma}
	
	\begin{proof}
		(a) The first assertion follows from the definition of a Cayley graph, and the fact that, for each $i\in\mathbf{I}$, 
		$G\rtimes L\leq \Aut(\Gamma_i)$ follows from \eqref{e:Aut} and the discussion above.
		
		(b) Again the form of $E\Gamma_i$ comes from the definition of $\Cay(G, S_i)$, and the discussion above shows that $E\Gamma_i$ is a $(G\rtimes L)$-orbit, so $\Gamma_i$ is normal edge-transitive.
	\end{proof}
	
	\section{Basic properties of Cayley codes}\label{s:cayc}
	
	Recall from Notation~\ref{not:def} that the ingredients for constructing a Cayley code are a Cayley graph $\Gamma=\Cay(G,S)=(G,E)$ and a linear code $B \leq \F^S$, where $S$ is a non-empty inverse-closed subset of $G\setminus\{1\}$. The corresponding Cayley code $\C(G,S,B)$, given by \eqref{e:cc}, is contained in $\F^E$. Recall that $\Aut(B)\leq \Sym(S)$ (in its induced action on $\F^S$), and for $\sigma\in A(G,S)$ (as in \eqref{e:Aut}) we denote by $\sigma^S$ the permutation induced by $\sigma$ on $S$. Recall also the subgroup $A(G,S,B)$ of $A(G,S)$ from   \eqref{e:Autcc}. 
	
	We verify several simple properties of Cayley codes, some of which can be found in \cite{KL2012} or \cite{KW2016}, while one strengthens a result in these papers (see Remark~\ref{rem:strong}).
	
	\begin{proposition}\label{p:cayc-props}
		Let $G, S, B, E, \F, \Gamma$ be as above. Then
		\begin{enumerate}
			\item[(a)]  $\C(G,S,B)$ is a linear code in $\F^E$ of length $|E|=|G|\cdot|S|/2$ and rate at least $2r(B)-1$, where $r(B)$ is the rate of $B$.
			
			\item[(b)] If $0\leq B_1\leq B_2\leq \F^S$, then $\C(G,S,B_1) \leq \C(G,S,B_2)$, and in particular, $\C(G,S,0)=0$, and $\C(G,S,\F^S)=\F^E$. 
			
			\item[(c)] Let $g\in G$ and $f\in\F^E$. Then for $\widehat{g}:\{x,sx\}\to \{xg,sxg\}$ (the induced action of $g$ on $E$), the map 
			\[
			g: f\to \widehat{g}\circ f
			\]
			defines a $G$-action on $\F^E$ such that $\C(G,S,B)$ is $G$-invariant, so $G\leq\Aut(\C(G,S,B))$.
			
			\item[(d)]  
			$\Aut(\C(G,S,B))\cap A(\Gamma) = G\rtimes A(G,S,B)$, with $A(\Gamma)$ as in \eqref{e:Aut} and  $A(G,S,B)$ as in \eqref{e:Autcc}. Further if,  in Notation~$\ref{not:L}$,  $L= A(G,S,B)$ and the partition of $S$ has only one part,  then $\Cay(G,S)$ is normal edge-transitive and $\C(G,S,B)$ is symmetric.
		\end{enumerate}
	\end{proposition}
	
	\begin{remark}\label{rem:strong}{\rm
			We note that the condition in Proposition~\ref{p:cayc-props}(d) on the partition from Notation~\ref{not:L} means that  $G\rtimes A(G,S,B)$ is edge-transitive on $\Cay(G,S)$ by Lemma~\ref{l:L}. Thus Proposition~\ref{p:cayc-props}(d) is a strengthening of \cite[Lemma 5]{KW2016} which assumes that $A(G,S,B)$ is transitive on $S$ (see also \cite[Proposition 5]{KL2012} and \cite[Lemma 5]{KW2010}).
		}
	\end{remark}

	\begin{proof}
		(a) Let $f,f'\in\C(G,S,B)$ and $a,a'\in\F$. Then by \eqref{e:cc}, for each $g\in G$, $\phi_g\circ f, \phi_g\circ f'\in B$, and hence also $\phi_g\circ (af+a' f') = a\phi_g\circ f+a' \phi_g\circ f'\in B$, since $B$ is linear.
		The length of $\C(G,S,B)$ is $|E|$ which equals $|E|=|G|\cdot|S|/2$. The assertion about the rate of $\C(G,S,B)$ is proved in \cite[Lemma 1]{KW2016}.
		
		(b) By \eqref{e:cc}, for each $i=1,2$, the code $\C(G,S,B_i)$ consists of all $f\in\F^E$ such that, for all $g\in G$, $\phi_g\circ f\in B_i$. All assertions now follow.
		
		(c) and (d) Proofs of these parts in the case where $A(G,S,B)$ is transitive on $S$ are given in \cite[Lemma 5]{KW2016}. However it is helpful to see a complete proof using our notation. Let $L=A(G,S,B)$.  First we note that, for each $x\in G\rtimes L$, $f\in \C(G,S,B)$, and $e\in E$,  the image $f^x=x\circ f$ is the map $f^x:e\to (e^x)f$. To verify that $f^x\in \C(G,S,B)$ we  show, for all $g\in G$, that the map $\phi_g\circ f^x$ lies in $B$.
		For each $s\in S$, we have by \eqref{e:phig}, 
		\[
		s(\phi_g\circ f^x) = (\{g,sg\})f^x = (\{g,sg\}^x)f.
		\]
		First, if $x\in G$, then $(\{g,sg\}^x)f= (\{gx,sgx\})f = s(\phi_{gx}\circ f)$, and so $\phi_g\circ f^x = \phi_{gx}\circ f$, and this lies in $B$ since $f\in\C(G,S,B)$. Next suppose that $x\in L\leq A(G,S)$ and note that $S^x=S$ by \eqref{e:Aut}, and  $B^x=B$ since $L\leq\Aut(B)$. Hence $(\{g,sg\}^x)f=  (\{g^x,s^xg^x\})f = s^x(\phi_{g^x}\circ f)$, so $\phi_g\circ f^x=x|_S \circ (\phi_{g^x}\circ f)$. Now $f':=\phi_{g^x}\circ f\in B$ by \eqref{e:cc} since $f\in\C(G,S,B)$, and  the restriction $x':=x|_S$ leaves $S$ invariant since $S^x=S$. Further, by definition, the composition $x|_S \circ (\phi_{g^x}\circ f) = x'\circ f' = (f')^{x'}$, and as $B^x=B$ and $f'\in B$ it follows that $(f')^{x'}\in B$. Thus  $\phi_g\circ f^x= (f')^{x'}\in B$. We have shown therefore that both $G$ and $L$ lie in $\Aut(\C(G,S,B))$, and hence also $G\rtimes L\leq \Aut(\C(G,S,B))$. 
		In particular, part (c) is proved, and so 
		\[
		\Aut(\C(G,S,B))\cap A(\Gamma) = G\rtimes (\Aut(\C(G,S,B)\cap A(G,S))
		\]
		and we have just proved that $\Aut(\C(G,S,B)\cap A(G,S)$ contains $L$. The reverse inclusion follows from \eqref{e:Autcc} and \eqref{e:cc}.
		Finally if the partition in Notation~\ref{not:L} for $L=A(G,S,B)$ has only one part then  $G\rtimes L$ is transitive on $E$, by Lemma~\ref{l:L}(b), and hence by definition, $\Cay(G,S)$ is normal edge-transitive and $\C(G,S,B)$ is symmetric.
	\end{proof}
	
	\section{Disconnected Cayley codes}\label{ss:disc}
	
	Recall the description in Lemma~\ref{lem:cay1} of the connected components of a disconnected Cayley graph $\Gamma=\Cay(G,S)=(G,E)$. The edge set $E$ is a disjoint union of the edges sets of the connected components, and this gives a $G$-invariant partition of $E$, and hence of $\F^E$ (Lemma~\ref{lem:Vdecomp}). We now show in Theorem~\ref{thm:disccaycode} that these decompositions lead to corresponding decompositions of Cayley codes. Theorem~\ref{thm:disconn} follows immediately from this result.
	
	\begin{theorem}\label{thm:disccaycode}
		Let $G$ be a finite group, $S$ an inverse-closed subset of $G\setminus\{1\}$ such that $H:=\la S\ra \ne G$, and let $\mathcal{T}$ be a set of coset representatives for $H$ in $G$ such that $1\in \mathcal{T}$. Let $B\leq \F^S$ be a linear code. Then, with $\Ecirc $ as in \eqref{e:EH} and $\phi_g$ as in \eqref{e:phig} and, for each $t\in \mathcal{T}$, setting
		\begin{equation}\label{e:cHt}
			\C(Ht,S,B):=\{ f\in\F^{\Ecirc  t} \mid\,  \forall\ g\in Ht,\ \phi_g\circ f\in B\},\quad  
		\end{equation}
		
		\begin{enumerate}
			\item[(a)] $\C(H,S,B)$ is the Cayley code corresponding to $\Cay(H,S)$ and $B$;
			
			\item[(b)]  for each $t\in \mathcal{T}$, $\C(Ht,S,B)$ is a linear code and, for the right multiplication action $t^{\Ecirc }:\Ecirc \to \Ecirc  t$, the map  $f\to t^{\Ecirc }\circ f$ defines an isomorphism $\C(Ht,S,B)\to \C(H,S,B)$;
			
			\item[(c)] the Cayley code $\C(G,S,B)=\bigoplus_{t\in \mathcal{T}} \C(Ht,S,B)$, a direct sum of $|\mathcal{T}|$ copies of $\C(H,S,B)$;
			
			\item[(d)] $\C(H,S,B)$ has the same rate and mminimum distance as $\C(G,S,B)$, while its relative distance is $|\mathcal{T}|$ times greater than  that of $\C(G,S,B)$. Also $\C(H,S,B)$ is symmetric if and only if $\C(G,S,B)$ is symmetric.
		\end{enumerate}
	\end{theorem}

	\begin{proof}
		(a) First we note that part (a) follows from the definition of a Cayley code in \eqref{e:cc}. 

		(b) We define a vector space isomorphism $\F^{\Ecirc t}\to \F^{\Ecirc}$ by $f\to t^{\Ecirc}\circ f$, for each $f\in\F^{\Ecirc t}$. To see that this map restricts to an isomorphism 
		$\C(Ht,S,B)\to \C(H,S,B)$ we consider  $f\in\C(Ht,S,B)\subseteq \F^{\Ecirc t}$. By \eqref{e:phig} and \eqref{e:cHt}, this is equivalent to the condition that,  for all  $ht\in Ht$,  $B$ contains
		\[
		\phi_{ht}\circ f = (\phi_h\circ t^{\Ecirc}) \circ f = \phi_h\circ (t^{\Ecirc} \circ f)
		\]
		and hence $t^{\Ecirc} \circ f\in \C(H,S,B)$ by \eqref{e:cHt}. Then since $t^{\Ecirc} \circ f$ runs over all elements of $\F^{\Ecirc}$ as $f$ runs over all elements of $\F^{\Ecirc t}$ it follows that  the map $f\to t^E\circ f$ takes $\C(Ht,S,B)$ into $\C(H,S,B)$. A similar argument shows that  $t^{E\circ}\circ f\in\C(H,S,B)$ implies  that $f\in\C(Ht,S,B)$, so $\C(Ht,S,B)$ is isomorphic to the Cayley code $\C(H,S,B)$, proving part (b). 
		
		(c) By Lemma~\ref{lem:cay1}, the $u:=|\mathcal{T}|$ connected components of $\Gamma=\Cay(G,S)$ are the graphs $(Ht,\Ecirc t)$, for $t\in \mathcal{T}$, with $\Ecirc, \Ecirc t$ as in \eqref{e:EH} and \eqref{e:Et}, and each $(Ht,\Ecirc t)$ is isomorphic to  the connected component $\Cay(H,S)=(H,\Ecirc)$ containing the identity $1\in \mathcal{T}$, via the maps $h\to ht,$ and $\{h,sh\}\to\{ht,sht\}$ on vertices and edges. 
		The edge set $E$ is therefore 
		the disjoint union $E=\cup_{t\in \mathcal{T}}\Ecirc t$, and this decomposition is in particular $G$-invariant. Thus, by Lemma~\ref{lem:Vdecomp}, the vector space $\F^E$ decomposes as $\F^E=\bigoplus_{t\in \mathcal{T}}\F^{\Ecirc t}$ and, for $f=(f_t)_{t\in \mathcal{T}}\in\F^X$ and $g\in G$, the $t^{th}$-component of $g^E\circ f$ is $f_{t'}$ where $Ht'=Htg$. More precisely, for a fixed $t\in \mathcal{T}$ and any $h\in H$,  the map $\phi_{ht}=\phi_h\circ t^E$ maps $s\to \{ht,sht\}$ (an edge of $\Ecirc t$), for $s\in S$. Moreover, for $f\in\F^E$, we have $\phi_{ht}\circ f = \phi_{ht}\circ (f|_{\Ecirc t})$, so if $f\in\C(G,S,B)$ then $\phi_{ht}\circ f\in B$ for all $h\in H, t\in \mathcal{T}$, and hence $f|_{\Ecirc t} \in\C(Ht,S,B)$ for each $t\in\mathcal{T}$.      It follows that $\C(G,S,B)\leq \bigoplus_{t\in\mathcal{T}} \C(Ht,S,B)$.

		To prove that equality holds, fix $t\in\mathcal{T}$ and let $f\in\C(Ht,S,B)$. Define $\widehat f\in\F^E$ by 
		\[
		(e) \widehat{f} = \left\{ \begin{array}{cc}
			(e)f& \text{if}\ e\in \Ecirc t \\
			0   & \text{otherwise.}
		\end{array} \right.
		\]
		Then $\widehat{f}\in\F^E$. For $ht\in Ht$ we have shown above that $\phi_{ht}\circ \widehat{f} = \phi_{ht}\circ (\widehat{f}|_{\Ecirc t})=  \phi_{ht}\circ f$ which, by \eqref{e:cHt}, lies in $B$. For all other elements $ht'\in Ht'$ with $t'\in \mathcal{T}\setminus\{t\}$, $\phi_{ht'}\circ \widehat{f} = \phi_{ht'}\circ (\widehat{f}|_{\Ecirc t'})= 0\in B$. Thus $\widehat{f}\in\C(G,S,B)$ and the $t^{th}$ component of $\widehat{f}$ is $f$. We conclude that  $\C(G,S,B)=\bigoplus_{t\in \mathcal{T}} \C(Ht,S,B) \cong \C(H,S,B)^u$, and part (c) is proved.
		
		(d) It follows from Proposition~\ref{p:cayc-props}(a) and part (c) that the rates $r_G$, $r_H$ of $\C(G,S,B)$ and $\C(H,S,B)$ satisfy
		\[
		r_G=\frac{2\cdot \dim(\C(G,S,B))}{|G|\cdot |S|}
		=\frac{2\cdot u\cdot \dim(\C(H,S,B))}{u\cdot |H|\cdot |S|} = r_H
		\]
		so $\C(G,S,B)$ and $\C(H,S,B)$ have the same rate. Further it follows from part (c) that a non-zero codeword in $\C(G,S,B)$ of minimum weight must lie in one of the direct summands and by part (b), $\C(H,S,B)$ must contain such a codeword, say of weight $d$. Hence $\C(H,S,B)$ and $\C(G,S,B)$ have the same minimum distance $d$, while the relative distances of $\C(H,S,B)$ and $\C(G,S,B)$ are, by Proposition~\ref{p:cayc-props}(a), $\frac{2d}{|H|\cdot |S|}$ and $\frac{2d}{|G|\cdot |S|}$ respectively. Finally, it follows from Proposition~\ref{p:cayc-props}(c) that $G$ acts transitively on the direct summands of $\C(G,S,B)$ in part (c), and hence $\C(H,S,B)$ is symmetric if and only if $\C(G,S,B)$ is symmetric.
	\end{proof}
	
	As a consequence of Theorem~\ref{thm:disccaycode}, when studying Cayley codes we usually assume that the Cayley graph is connected.
	
	\section{Different decompositions of Cayley codes}\label{s:decomp}
	
	In this section we discuss a process to identify a family of symmetric Cayley codes  related to a given Cayley code, which  allows us to embed the given Cayley code into a direct sum of symmetric Cayley codes (using Proposition~\ref{p:cayc-props}(b)). However the given Cayley code itself only admits a decomposition if the input code $B$ admits a corresponding decomposition. The  process discussed in this section does not necessarily preserve Cayley graph connectivity, so one should perform this decomposition to symmetric Cayley codes first, if available, and then reduce to the connected case using Section~\ref{ss:disc}.
	
	We will eventually use the natural decomposition of the Cayley graph $\Cay(G,S)$ given in Notation~\ref{not:L} and Lemma~\ref{l:L} corresponding to the $(G\rtimes L)$-orbits on edges of the associated Cayley graph, but first we consider a more general decomposition.
	
	\begin{hypothesis}\label{hyp1}
		Let $G$ be a finite group, let  $S$ a non-empty inverse-closed subset of $G\setminus\{1\}$, let $\F$ be a finite field, and let  $B\leq \F^S$, a linear code.  
		\begin{enumerate}
			\item[(a)] 
			Suppose that $S$ is the disjoint union  $S=\cup_{i\in\mathbf{I}}S_i$, where each $S_i=S_i^{-1}$ is inverse-closed. 
			Then  $\F^S=\oplus_{i\in\mathbf{I}} \F^{S_i}$, and  for each $i\in \mathbf{I}$, if $B_i$ is the restriction of $B$ to $\F^{S_i}$, then  $B$ is  subdirect in $ \oplus_{i\in\mathbf{I}}B_i \leq \F^S$ (projects onto each direct summand).
			\item[(b)]
			Let $\Gamma=\Cay(G,S) = (G,E)$ and let $\Gamma_i=\Cay(G,S_i)$, for $i\in\mathbf{I}$, so $E\Gamma=E$ is the disjoint union  $\cup_{i\in\mathbf{I}} E\Gamma_i$  and $\F^E=\oplus_{i\in\mathbf{I}} \F^{E\Gamma_i}$.  
		\end{enumerate}
	\end{hypothesis}

	\begin{proposition} \label{prop:SiBi}
		Suppose that Hypothesis~\ref{hyp1} holds. 
		Then  the Cayley code $\C(G,S,B)\leq \C(G,S,\oplus_{i\in\mathbf{I}}B_i)$ and $\C(G,S,\oplus_{i\in\mathbf{I}}B_i)=\oplus_{i\in\mathbf{I}} \C(G,S_i,B_i)$ holds. \end{proposition}
	
	\begin{remark}\label{r:SiBi}
		{\rm Certainly $G, S, B$ and the disjoint union $S=\cup_{i\in\mathbf{I}}S_i$ determine the $B_i$ and hence the direct decomposition $\oplus_{i\in\mathbf{I}} \C(G,S_i,B_i)$. However it is not clear to us whether the inclusion $\C(G,S,B)\leq \oplus_{i\in\mathbf{I}} \C(G,S_i,B_i)$ is necessarily subdirect in general.
		}    
	\end{remark}

	\begin{proof}
		By Proposition~\ref{p:cayc-props}(b), $\C(G,S,B)\leq \C(G,S,\oplus_{i\in\mathbf{I}}B_i)$. Next we verify the  direct decomposition. Let $f\in\C(G,S,\oplus_{i\in\mathbf{I}}B_i)$. Then $f\in\F^E=\oplus_{i\in\mathbf{I}} \F^{E\Gamma_i}$, so $f=\oplus_{i\in\mathbf{I}}f_i$ where $f_i = f|_{E\Gamma_i}$. The defining condition for membership of $\C(G,S,\oplus_{i\in\mathbf{I}}B_i)$ is that, for all $g\in G$, $\phi_g\circ f\in \oplus_{i\in\mathbf{I}}B_i$ with $\phi_g$ as in \eqref{e:phig}. For each $i$, the restriction $\phi_{i,g}:=\phi_g|_{S_i}$ has image 
		$(S_i)\phi_{i,g}=(S_i)\phi_g = \{ \{g,sg\}\mid s\in S_i\} \subseteq E\Gamma_i$, and hence
		\[
		f_i\in\F^{E\Gamma_i}\quad \mbox{with}\quad \phi_{i,g}\circ f_i=(\phi_g\circ f)|_{S_i} \in \F^{S_i}.
		\]
		Moreover $\phi_g\circ f\in \oplus_{i\in\mathbf{I}} B_i$, and hence   $\phi_{i,g}\circ f_i=(\phi_g\circ f)|_{S_i} \in B_i$. Since this holds for all $g\in G$, we have $f_i\in  \C(G,S_i,B_i)$, and since this holds for all $i\in\mathbf{I}$, it follows that $\C(G,S,\oplus_{i\in\mathbf{I}}B_i)\leq \oplus_{i\in\mathbf{I}} \C(G,S_i,B_i)$.  On the other hand, for each $i$ and each $f_i\in \C(G,S_i,B_i)$, let $\widehat{f_i}:E\to\F$ be the function such that the restriction of $\widehat{f_i}$ to $E\Gamma_j$ is the constant zero function if $j\ne i$, and is equal to $f_i$ if $j=i$. Then for each $g\in G$ we have
		$\phi_g\circ \widehat{f_i} \in B_i\leq \oplus_{i\in\mathbf{I}}B_i$, and hence $\widehat{f_i}\in\C(G,S,\oplus_{i\in\mathbf{I}}B_i)$. 
		Thus we have proved that $\C(G,S,\oplus_{i\in\mathbf{I}}B_i)=\oplus_{i\in\mathbf{I}} \C(G,S_i,B_i)$ holds.
	\end{proof}
	
	We now deduce Theorem~\ref{t:SiBi} from this general result.
	
	\bigskip\noindent
	\emph{Proof of Theorem~\ref{t:SiBi}.}\quad 
	Let $\Gamma=\Cay(G,S)$ with $S$ a non-empty inverse closed subgroup of $G\setminus\{1\}$, let $B\leq \F^S$, and
	let $L=A(G,S,B)$ as in \eqref{e:Autcc}. Then $G, S, L$ determine the partition $S=\cup_{i\in\mathbf{I}}S_i$as in Notation~\ref{not:L}. By Lemma~\ref{l:L}(b), $\Gamma$ is the edge disjoint union of its Cayley subgraphs $\Gamma_i=\Cay(G,S_i)$, $G\rtimes L$ acts as automorphisms on each of them and is transitive on each $E\Gamma_i$, so each $\Gamma_i$ is normal edge-transitive. Also it follows from Proposition~\ref{p:cayc-props}(d) that each $\C(G, S_i,B_i)$ is symmetric, proving the first assertion. Now assume that $B=\oplus_{i\in\mathbf{I}}B_i$. Then by Proposition~\ref{prop:SiBi}, $\C(G,S,B)=\oplus_{i\in\mathbf{I}} \C(G,S_i,B_i)$, completing the proof. 
	
	\section{Some examples of Cayley codes}\label{s:ex}
	
	In this section we present several families of Cayley codes, where we specify certain possibilities for the local input code $B$ (Subsection~\ref{ss:repaug}) or the Cayley graph $\Cay(G,S)$ (Subsection~\ref{ss:cycle}).  
	
	\subsection{$B$ the repetition code or augmentation code}\label{ss:repaug}
	
	For a finite set $S$,  the \emph{repetition code} $R(S)<\mathbb{F}^S$ consists of all constant functions $c\in\F$, for $c\in\F$, where  $f_c:x\to c$ for $x\in S$. Also the \emph{augmentation code} $A(S) <\mathbb{F}^S$ is the co-dimension $1$ subspace $A(S):=\{f\in\F^S\mid \sum_{x\in S}(x)f=0\}$. Both $R(S)$ and $A(S)$ are invariant under the full symmetric group $\Sym(S)$, and hence may be chosen as the local input code $B$ for any Cayley graph.

	\begin{lemma}\label{lem:rep}
		Let $\Cay(G,S)=(G,E)$, $\F, B$ be as in Notation~\ref{not:def}, and let $R(S), A(S)$ be as above. Then
		\begin{enumerate}
			\item[(a)] If $B=R(S) < \F^S$ and $\Cay(G,S)$ is connected, then $\C(G,S,B)$ is the repetition code  $R(E) <\F^E$.
			\item[(b)] If $B=A(S)$ and $\F$ has odd characteristic, then $\C(G,S,B)\leq A(E)$,  the augmentation code in $\F^E$.
		\end{enumerate}
	\end{lemma}
	
	\begin{proof}
		(a) Suppose that $B=R(S)$. Then it follows from \eqref{e:cc} that $R(E)\leq \C(G,S,B)$. Conversely let $f\in\C(G,S,B)$ and $s\in S$.  Then  $e_0=\{1,s\}\in E$, and we let $c=(e_0)f$, and $e\in E$ be an arbitrary edge with $e\ne e_0$. We will show that $(e)f=c$, and from this it follows that $f=f_c\in R(E)$, proving that $\C(G,S,B)=R(E)$. Since $\Cay(G,S)$ is connected there is a path (an edge-sequence) $ e_0,e_1,\dots,e_t$ in $\Cay(G,S)$   such that $e_t=e$ and consecutive edges are incident. Choose a path with $t$ minimal, so $t\geq 1$ since $e\ne e_0$. Also, for each $i<t$, the edges $e_i $ and $e_{i+1}$ are distinct by the minimality of $t$ and we let   $\{g_i\} = e_i\cap e_{i+1}$. We will prove by  induction on $i\leq t$ that  $(e_i)f=(e_{i-1})f=\dots =(e_0)f= c$. If $i=1$ then both $e_0$ and $e_1$ are incident with $g_0$ and by \eqref{e:cc}, $\phi_{g_0}\circ f \in B=R(S)$, so $f$ is constant on the edges incident with $g_0$. Thus $(e_1)f=(e_0)f=c$, and we are finished if $t=1$. Now assume that $t>1$, that $1<i\leq t$, and assume inductively that the claim holds for $i-1$, that is to say, $(e_{i-1})f=\dots =(e_0)f= c$. Now  $\{g_{i-1}\} = e_{i-1}\cap e_{i}$ and, again by \eqref{e:cc}, $\phi_{g_{i-1}}\circ f \in B=R(S)$, so $f$ is constant on the edges incident with $g_{i-1}$. Thus $(e_i)f=(e_{i-1})f=c$, and by induction the claim holds for all $i\leq t$. Thus $(e_t)f=c$, and as discussed above, part (a) follows.

		(b) Let $f\in\C(G,S,B)$. Then, for all $g\in G$, $\phi_g\circ f\in A(S)$. This implies that $\sum_{e\in E_g}(e)f=0$, where $E_g$ is as in \eqref{e:eg}. Now as $g$ ranges over $G$, the sets $E_g$ cover each edge exactly twice, since an edge $\{x,y\}$ lies precisely in the edge sets $E_x$ and $E_y$. Thus  
		\[
		0=\sum_{g\in G} \left(\sum_{e\in E_g}(e)f\right) = 2 \sum_{g\in G} (e)f
		\]
		and since $\F$ has odd characteristic, this implies that $\sum_{g\in G} (e)f=0$, so $f\in A(E)$.
	\end{proof}
	
	The Cayley code in part (b) above can be significantly smaller than the augmentation code. We give a family of examples in Subsection~\ref{ss:cycle} for which  this is the case with the Cayley graph being a cycle of arbitrary length.

	\subsection{Cayley codes for cycles}\label{ss:cycle}
	
	For $n\geq3$, the cycle $\Gamma=\C_n$ of length $n$ has vertex set  $V\Gamma=\mathbb{Z}_n=\{0,1,\dots, n-1\}$, which we often view as the additive group of integers modulo $n$, and edge-set $E=E\Gamma=\{e_0,\dots,e_{n-1}\}$, where $e_i=\{i-1,i\}$ for each $i$.  
	It is a Cayley graph for the additive group $G=\mathbb{Z}_n$ relative to the inverse-closed generating set $S=\{-1, 1\}$. Thus $\Gamma =\Cay(G,S)$.  
	Note that, for  $i \in \mathbb{Z}_n$, the map $\phi_s$ from \eqref{e:phig} is $\phi_i: s \mapsto \{i,i+s\}$ where $s \in \{-1,1\}$. 
	
	To construct a
	Cayley code for $\Gamma$, we require a linear code $B \leq \mathbb{F}^S \cong
	\mathbb{F}^2$. By Proposition~\ref{p:cayc-props}(b), $\mathbf{C}(G,S,B)$ is the zero code if $B=0$ and is the complete code $\mathbb{F}^E$ if $B=\mathbb{F}^S$, so our interest is in the remaining cases where $B$  is a one-dimensional subspace of $\mathbb{F}^S$. For $\mathbb{F}=\mathbb{F}_q$ there are exactly $q+1$ such subspaces, and we label them as 
	\[
	B=B_a = \langle b_a\rangle\ \textbf{for $a\in\mathbb{F}_q\cup \{\infty\}$, where}\ 
	b_\infty: -1\to 0\  \textbf{and }\ 1\to 1,
	\]
	and for $a\in\mathbb{F}_q$,  $b_a:-1\to 1$ and $1\to a$. The corresponding Cayley codes are   
	\begin{equation}\label{e:Ca}
		\mathbf{Cyc}_a :=\mathbf{C}(G,S,B_a) = \{f \in \mathbb{F}^E \, | \, \phi_g \circ f \in B_a
		\textrm{ for all } g \in G\}.     
	\end{equation}
	We have the following possibilities for these Cayley codes.
	
	\begin{proposition}\label{lem:cycle1}
		Let $\Gamma, n, G, B_a, \mathbf{Cyc}_a$ be as above. Then 
		\begin{enumerate}
			\item[(a)] if either $a=\infty$ or $a^n\ne 1$, then $\mathbf{Cyc}_a=0$;
			\item[(b)] if $a\in\mathbb{F}_q$ and $a^n=1$, then $\mathbf{Cyc}_a=\langle f_a\rangle$ where $(e_i)f_a=a^i$ for $i\in\mathbb{Z}_n$.
		\end{enumerate}
	\end{proposition}
	
	\begin{proof}
		By the definition of a Cayley code in \eqref{e:cc}, an element $f\in\F^E$ lies in $\mathbf{Cyc}_a$ if and only if $\phi_i\circ f\in B_a$ for all $i\in G=\mathbb{Z}_n$ where $\phi_i:s\to \{i, i+s\}$ (for $s\in S=\{-1,1\}$) as in \eqref{e:phig}. 
		
		Suppose first that $a\not\in\{\infty,0\}$. Then the only element $b\in B_a$ such that the image $(S)b$ contains $0$ is the zero element $0\in\F^S$. Suppose first that $\phi_0\circ f$ is the zero element $0\in \F^S$, that is to say,  $(e_0)f=(e_1)f=0$. Then $(S)(\phi_1\circ f)$ contains $(e_1)f=0$ and hence $\phi_1\circ f$ lies in $B_a$ if and only if $\phi_1\circ f = 0\in\F^S$. An inductive argument using the same reasoning yields that $f\in\mathbf{Cyc}_a$ if and only if $\phi_i\circ f = 0$ for all $i$, and hence $f=0\in\F^E$.  Now suppose that $\phi_0\circ f$ is a non-zero element of $\F^S$. A necessary condition for $f$ to lie in $\mathbf{Cyc}_a$ is that   $\phi_0\circ f \in B_a$, and this holds if any only if $\phi_0\circ f$ is a scalar multiple of the element $b_a$ defined above, and without loss of generality we  assume that $\phi_0\circ f=b_a$. Note that  $\phi_0\circ f=b_a$ is equivalent to $(e_0)f= 1, (e_1)f=a$. In particular $(-1)(\phi_1\circ f)=(e_1)f=a$, and hence $\phi_1\circ f \in B_a$ if and only if $\phi_1\circ f =a\, b_a$ (equivalently $(1)(\phi_1\circ f)=(e_2)f=a^2$). This time the same inductive argument shows that, for all $i\in G$, $\phi_i\circ f \in B_a$ if and only if $\phi_i\circ f =a^i b_a$ (equivalently $(e_i)f=a^i, (e_{i+1})f=a^{i+1}$). In order for this condition to hold for $i=n-1$ we require $(e_0)f=(e_{(n-1)+1})f=a^{(n-1)+1}=a^n$, but we already have $(e_0)f=1$. Thus, if $a^n\ne 1$ then the code $\mathbf{Cyc}_a=0$, while if $a^n=1$ then $\mathbf{Cyc}_a=\langle f_a\rangle$ with $f_a$ as in part (b). 
		
		It remains to consider the cases $a\in\{\infty, 0\}$. The arguments for these cases are similar. Suppose first that $f\in\mathbf{Cyc}_a$ and that $\phi_0\circ f = 0\in \F^S$, that is, $(e_0)f=(e_1)f=0$. If $a=1$ then the only element $b\in B_a$ with $(-1)b=0$ is the zero element of $\F^S$, and the argument in the previous paragraph shows that $\phi_i\circ f = 0\in \F^S$ for all $i$ so that $f=0\in\F^E$. Similarly if $a=\infty$ then the only element $b\in B_a$ with $(1)b=0$ is  $0\in \F^S$, and again we find that $\phi_i\circ f = 0$ for all $i$ and $f=0\in\F^E$. Now assume that $f\in\mathbf{Cyc}_a$ and that $\phi_0\circ f$ is a non-zero element of $B_a$ (a necessary condition for $f\in\mathbf{Cyc}_a$), so $\phi_0\circ f$ is a scalar multiple of $b_a$, and without loss of generality we  assume that $\phi_0\circ f=b_a$. Note that  $\phi_0\circ f=b_a$ is equivalent to $(e_0)f= 1, (e_1)f=0$ if $a=0$, and to $(e_0)f= 0, (e_1)f=1$ if $a=\infty$. If $a=0$ then $(-1)(\phi_1\circ f)=(e_1)f=0$, and as the only element $b\in B_a$ with $(-1)b=0$ is the zero element, we have by \eqref{e:Ca} that $\phi_1\circ f=0\in\F^S$ and, in turn, the usual inductive argument shows that $f=0\in \F^E$, which is a contradiction. Thus $\mathbf{Cyc}_0=0$. Finally suppose that $a=\infty$.  Then the only element $b\in B_\infty$ with $(1)b=0$ is the zero element, and since $(1)\phi_{n-1}\circ f = (e_0)f=0$, it follows from \eqref{e:Ca} that $\phi_{n-1}\circ f=0\in\F^S$. The usual inductive reasoning shows that $\phi_{n-i}\circ f=0\in\F^S$ for all $i$, and hence that $f=0\in\F^E$, whence  $\mathbf{Cyc}_\infty=0$,  completing the proof. 
	\end{proof}
	
	\begin{remark}\label{r:cycle}
		{\rm 
			(a)  Note that, taking $a=1$, the input code $B_1$ is the repetition code $R(S)$, and since  $a^n=1^n=1$, it follows from Proposition~\ref{lem:cycle1}(b) that the corresponding Cayley code $\mathbf{Cyc}_1=\langle b_1\rangle$, which is the repetition code in $\F^E$, in agreement with Lemma~\ref{lem:rep}(a).
			
			\medskip\noindent
			(b) Similarly, taking $a=-1$, the input code $B_{-1}$ is the augmentation code $A(S)$, and if either $n$  or $q$ is even then  $a^n=(-1)^n=1$, so by  Proposition~\ref{lem:cycle1}(b) the corresponding Cayley code $\mathbf{Cyc}_{-1}=\langle b_{-1}\rangle$ is a one-dimensional subspace of the augmentation code $A(E)$ in $\F^E$, giving examples for proper inclusion in Lemma~\ref{lem:rep}(b). 
			
			\medskip\noindent
			(c) If $n$ is even, then $\Gamma=\mathbf{C}_n$ is also a Cayley graph for the dihedral group $D=\langle \sigma^2, \tau\sigma\rangle$ of order $n$, where 
			\[
			\sigma:i\to i+1\quad \text{and}\quad \tau: i\to -i, \quad \text{for $i\in \mathbb{Z}_n$}
			\]
			since $D$ is a  subgroup of $\Aut(\Gamma)=\langle \sigma, \tau\rangle\cong D_{2n}$ acting regularly on vertices. Namely $\Gamma=\Cay(D,S')$ for the generating set $S'=\{\sigma\tau, \tau\sigma\}$ of $D$, and we label the edges in the edge-set $E'$ as  $e_0'=\{1,\sigma\tau\}=\{ 1, \tau\sigma^{-1}\}$, $e_1'=\{1, \tau\sigma\}$ and for  $i=1,\dots,(n-2)/2$, 
			\[
			e_{2i}'= (e_0'){\sigma^{2i}} = \{ \sigma^{2i}, \tau\sigma^{2i-1}\},\quad e_{2i+1}'= (e_1'){\sigma^{2i}} = \{ \sigma^{2i}, \tau\sigma^{2i+1}\}.
			\]
			A similar analysis to that given in Proposition~\ref{lem:cycle1} yields pairwise distinct one-dimensional (rank 1) Cayley codes 
			$\mathbf{C}(D, S', B_a')=\langle f_a'\rangle$ for distinct $a\in\F\setminus\{0\}$. Here the  input code $B_a'=\langle b_a'\rangle$ where $b_a': \sigma\tau\to 1, \tau\sigma\to a$, and the  generator $f_a'\in\F^{E'}$ maps $e'_{2i}\to 1$ and $e'_{2i+1}\to a$ for $i=0,\dots,(n-2)/2$. 
			
		}
	\end{remark}

	\section{Products of graphs and Cayley codes}\label{s:prod}
	
	Several product constructions for graphs take the vertex set to be the Cartesian product of the vertex sets of  two input graphs. We consider two such constructions defined as follows, and then discuss implications for Cayley codes under these graph constructions. 
	
	\begin{definition}\label{def:prod}
		{\rm 
			Let $\Gamma=(V\Gamma, E\Gamma)$ and $\Sigma=(V\Sigma, E\Sigma)$ be graphs. For each of the product constructions below the vertex set  is $V\Gamma\times V\Sigma$. 
			
			\begin{enumerate}
				\item[(a)] The \emph{(direct) product} $\Gamma\times\Sigma$ has edge set (see \cite[Section 6.3]{GR}): 
				\[
				E(\Gamma\times\Sigma)=\{ \{(\alpha,\sigma),(\alpha',\sigma')\} \mid \{\alpha, \alpha'\}\in E\Gamma\ \mbox{and}\ \{\sigma,\sigma'\}\in E\Sigma\}.
				\]
				
				\item[(b)] The \emph{Cartesian product} $\Gamma\,\square\,\Sigma$ has edge set 
				\cite[Section 7.14]{GR}: 
				\[
				E(\Gamma\,\square\,\Sigma)=\{ \{(\alpha,\sigma),(\alpha',\sigma')\} \mid 
				\begin{array}{c}
					\alpha= \alpha'\  \mbox{with}\ \{\sigma,\sigma'\}\in E\Sigma,\  \\ 
					\mbox{or}\ \{\alpha, \alpha'\}\in E\Gamma\  \mbox{with}\ \sigma=\sigma' 
				\end{array} 
				\}.
				\]
			\end{enumerate}
		}
	\end{definition}
	
	Let $\Delta$ be $\Gamma\,\times\,\Sigma$ or $\Gamma\,\square\,\Sigma$ as in Definition~\ref{def:prod}. Then $\Aut(\Delta)$ contains $\Aut(\Gamma)\times\Aut(\Sigma)$ in its natural product action, that is,  
	$(g,h):(\alpha,\sigma)\to (\alpha^g, \sigma^h)$ for  $(g,h)\in \Aut(\Gamma)\times\Aut(\Sigma)$ and $(\alpha,\sigma)\in V\Delta$. In particular, if $G\leq \Aut(\Gamma)$ and $H\leq \Aut(\Sigma)$ with $G, H$ regular on $V\Gamma, V\Sigma$, respectively, then $G\times H$ is regular on $V\Delta$, and consequently, for Cayley graphs $\Gamma = \Cay(G,S)$ and $\Sigma=\Cay(H,T)$, the product graph $\Delta$ is a Cayley graph for $G\times H$, see \cite[Lemmas 3.7.1 and 3.7.2]{GR}: the `joining sets' are $S\times T$ for the direct product, and $S\dot\cup T$ for the Cartesian product (identifying $S, T$ with subsets of $G\times H$). More generally, if $\Gamma$ and $\Sigma$ are both vertex-transitive, then also $\Delta$ is vertex-transitive.
	Also edge-transitivity of $\Gamma, \Sigma$ guarantees edge-transitiv\-ity of $\Gamma \times \Sigma$, but in general does not guarantee this for the Cartesian product $\Gamma\,\square\,\Sigma$. Thus  Cayley codes for direct products of Cayley graphs may have more desirable properties than those for Cartesian products.  We now introduce the Cayley codes for these products where  $\Gamma$ and $\Sigma$ are both Cayley graphs. 
	
	\subsection{Cayley codes for Cartesian products}\label{s:cartprod}
	We consider Cayley graphs $\Gamma = \Cay(G,S)$ and $\Sigma=\Cay(H,T)$ for finite groups $G, H$ with inverse-closed generating sets $S, T$ of $G, H$, respectively. For notational simplicity we identify the subset $S$ of $G$  with the subset $S\times 1=\{ (s,1_H) \mid s \in S \}$ of $G \times H$, and similarly we identify $T$ with  $1\times T=\{ (1_G,t)  \mid t \in T \}$, and note that each of these  is an inverse-closed subset of $G\times H$. Consider the   $\Delta=\Cay(G \times H, S \dot\cup T)$. Then, as discussed above, $\Delta$ is the Cartesian product $\Gamma \,\square\,\Sigma$ (as in Definition~\ref{def:prod}(b)).
	The generating set $S\dot\cup T$ for this Cayley graph is a disjoint union  as in Hypothesis~\ref{hyp1}(a), and so, for a finite field $\F$ and linear codes $A\leq \F^S$ and $B\leq \F^T$, it follows from Proposition~\ref{prop:SiBi} that the Cayley code $\C(G \times H,S \dot\cup T, A \oplus B)$ has the following decomposition:
	\begin{equation}\label{e:prod1}
		\C(G \times H,S \dot\cup T, A \oplus B)= \C(G \times H , S, A) \oplus \C(G \times H , T, B).     
	\end{equation}
	We make a further analysis of the structure of the two direct summands and prove the following.
	
	\begin{proposition}\label{p:cartesianprod}
		With the notation as above.
		\begin{enumerate}
			\item[(a)] $\C(G \times H , S, A) \cong \C(G,S,A) \otimes \mathbb{F}^{|H|}$;
			\item[(b)] $\C(G \times H , T, B) = \mathbb{F}^{|G|} \otimes \C(H,T,B)$;
		\end{enumerate}
		and the assertion of Theorem~$\ref{t:cprod}$ is valid.
	\end{proposition}
	
	For the proof we use a natural map linking the edge-set of $\Delta$ with those for $\Gamma$ and $\Sigma$. For $A\Delta, A\Gamma, A\Sigma$ the sets of arcs of $\Delta, \Gamma, \Sigma$,  define a natural bijection $\psi \colon A\Delta\to (A\Gamma \times H) \cup (G \times A\Sigma)$ by
	\[
	\psi: (\, (g,h),\ (sg, h)\,) \to  ( \,(g,sg),\ h \,); \quad
	(\, (g,h),\ (g, th)\,) \to  ( \, g,\ (h,th) \,).
	\]
	This induces a bijection  
	$\hatpsi: E\Delta\to (E\Gamma \times H) \dot\cup (G \times E\Sigma)$, namely 
	\begin{equation}\label{e:psihat}
		\hatpsi: \{(g,h),\, (sg, h)\}\to 
		( \,\{g,sg\},\, h \,); \quad \{(g,h),\, (g, th)\}\to 
		( \, g ,\, \{h,th\} \,).   
	\end{equation}

	\begin{proof}
		We give the proof details for part (a), since the proof of part (b) is entirely similar. The group $G_0 :=  G \times \{ 1_H \}$ has index $u := |H|$ in $G \times H$, and $\mathcal{T} := \{ (1_G,h) \mid h \in H \}$ is a set of right coset representatives for $G_0$ in $G\times H$. By Lemma~\ref{lem:cay1}(a), 
		\[
		E_0=\{ \{(g,1_H), (sg,1_H)\} | (g,1_H)\in G_0, s\in S\}  
		\]
		is the edge-set, as in \eqref{e:EH}, for the connected component of $\Cay(G \times H, S)$ containing the identity, and there are exactly $u$ connected components. It follows from Theorem~\ref{thm:disccaycode}(c) that 
		\begin{equation}\label{e:prod2}
			\C(G \times H , S, A) = \oplus_{t \in \mathcal{T}} \, \C(G_0t, S, A)     
		\end{equation}
		is isomorphic to a direct sum of $u$ copies of $\C(G_0, S, A)$ and this, in turn, is isomorphic to $\C(G,S,A) \otimes \mathbb{F}^{|H|}$. The last assertion uses the natural bijection $\hatpsi$ given in \eqref{e:psihat}. Thus part (a) is proved, and part (b) follows by an identical argument. Theorem~\ref{t:cprod} follows from parts (a) and (b) and \eqref{e:prod1}.
	\end{proof}
	
	Finally we deduce information about the parameters (see Table \ref{tab1}) of the Cartesian product cayley code from those of the factor codes.
	
	\begin{remark}\label{r:cprod}
		{\rm
			With the notation as above let $n = |G|$, $n' = |H|$, $\ell = |S|$ and $\ell' = |T|$, and suppose that $\C(G,S,A)$ and $\C(H,T,B)$ have rank $k, k'$ and distance $d, d'$, respectively. We discuss below the equivalent parameters for the Cartesian product Cayley code $\C(G \times H,S \dot\cup T, A \oplus B)$.
			\begin{center}
				\begin{tabular}{c c c c} 
					Parameter  & $\C(G,S,A)$ & $\C(H,T,B)$ & $\C(G \times H,S \dot\cup T, A \oplus B)$  \\ [0.5ex] 
					\hline 
					\text{Length} & $n\ell/2$ & $n'\ell'/2$ & $nn'(\ell+\ell')/2$  \\ 
					\text{Rank} & $k$ & $k'$ & $kn'+k'n$ \\
					\text{Distance} & $d$ & $d'$ & $ \operatorname{min}\{d,d'\}$ \\
					\text{Rate} & $2k/n\ell$ & $2k'/n'\ell'$ & $2(kn'+k'n)/(nn'(\ell+\ell'))$ \\
					\text{Relative distance} & $2d/n\ell$ & $2d'/n'\ell'$ & $ 2\operatorname{min}\{d,d'\}/(nn'(\ell+\ell'))$ \\
					\hline
				\end{tabular}
			\end{center}
			
			The length follows  from the definition of a Cartesian product in Definition \ref{def:prod} (b), and the rank follows  from Theorem \ref{t:cprod}, noting that $\operatorname{Dim}(\C(G,S,A) \otimes \mathbb{F}^{|H|}) = |H| \cdot r = n'r$ and $\operatorname{Dim}(\mathbb{F}^{|G|} \otimes \C(H,T,B)) = |G| \cdot r' = n r'$. A nonzero codeword of $\C(G \times H,S \dot\cup T, A \oplus B)$ of minimum weight must lie in one of the direct summands of \eqref{e:prod1}, say in  $\C(G,S,A) \otimes \mathbb{F}^{|H|}$, and such a minimum weight codeword must lie in one of the direct summands of \eqref{e:prod2}. Since  the direct summands in \eqref{e:prod2} are isomorphic codes, the weight of such a codeword would equal the minimum weight $d$ of a codeword of  $\C(G,S,A)$. The entries for the rate and relative distance follow on dividing the rank and distance by the length.
		}
	\end{remark}

	\subsection{Cayley codes for direct products}\label{s:dirprod}
	
	As in Subsection \ref{s:dirprod} we consider Cayley graphs $\Gamma = \Cay(G,S)$ and $\Sigma=\Cay(H,T)$ for finite groups $G, H$ with inverse-closed generating sets $S, T$ of $G, H$, respectively, and we identify 
	the $S$ and $T$ with the subsets $S\times 1$ and $1\times T$ of $G \times H$. Also, as discussed above, the graph $\Delta=\Cay(G \times H, S \times T)$ is the direct product $\Gamma \times\Sigma$ defined in Definition~\ref{def:prod}(a). The number of edges of $\Delta$ is $\frac{|G\times H|\cdot|S\times T|}{2} = 2\cdot |E\Gamma|\cdot |E\Sigma|$, so the length of a Cayley code for $\Delta$ is twice the product of the lengths of Cayley codes for $\Delta$ and $\Sigma$. We show that the Cayley codes for $\Delta$ are linked with tensor products of Cayley codes for $\Delta$ and $\Sigma$. 
	First we define a natural surjection $\psi$ from $E\Delta$ to $E\Gamma\times E\Sigma$, noting that each edge of $\Delta$ is of the form 
	\begin{equation}\label{eq:e}
		\text{$e_{g,h,s,t}=\{(g,h),(sg,th)\}$, for $g\in G, h\in H, s\in S, t\in T$.}
	\end{equation}
	
	\begin{lemma}\label{l:e}
		Define $\psi: E\Delta\to E\Gamma\times E\Sigma$, by $(e_{g,h,s,t})\psi =  
		( \,\{g,sg\},\, \{h,th\} \,)$, with 
		$e_{g,h,s,t}$ as in \eqref{eq:e}. Then,
		\begin{enumerate}
			\item[(a)]  for $g\in G, h\in H, s\in S, t\in T$, the following conditions are equivalent.
			\begin{enumerate}
				\item[(i)] $e_{g,h,s,t}=e_{g',h',s',t'}$;
				\item[(ii)] either $(g',h',s',t')=(g,h,s,t)$ or $(g',h',s',t')=(sg,th,s^{-1},t^{-1})$;
			\end{enumerate}
			\item[(b)] $\psi$ is a well-defined surjection.
			\item[(c)] If $(e_{g,h,s,t})\psi=(e_{g',h',s',t'})\psi$, then either $e_{g',h',s',t'}=e_{g,h,s,t}$, or 
			$e_{g',h',s',t'}=e_{sg,h,s^{-1},t}=e_{g,th,s,t^{-1}}$. Thus $\psi$ is a $2$-to-$1$ map.
		\end{enumerate}
	\end{lemma}
	
	\begin{proof}
		(a) Suppose that (i) holds.  Then  $(g',h')\in \{(g,h),(sg,th)\}$. If $(g',h') = (g,h)$, then also $(s'g',t'h') = (sg,th)$ by part (i), and we have $(g',h',s',t')=(g,h,s,t)$. On the other hand, if  $(g',h') = (sg,th)$, then also $(s'g',t'h') = (g,h)$ by part (i), and it follows that $s'=s^{-1}$ and $t'=t^{-1}$, so $(g',h',s',t')=(sg,th,s^{-1},t^{-1})$. 
		
		Conversely, suppose that (ii) holds. For the first alternative part (i) requires no proof, while for the second alternative we have $s'g'=s^{-1}(sg)=g$ and $t'h'=t^{-1}(th)=h$, and part (i) holds in this case also.
		
		(b) To see that $\psi$ is well defined, suppose that (a)(i) holds. Then also (a)(ii) holds, and either alternative implies that $\{g,sg\}=\{g',s'g'\}$ and $\{h,th\}=\{h',t'h'\}$, so $(e_{g,h,s,t})\psi=(e_{g',h',s',t'})\psi$. Thus $\psi$ is well defined. Clearly $\psi$ is surjective.
		
		(c) Suppose that $(e_{g,h,s,t})\psi=(e_{g',h',s',t'})\psi$.  Then $\{g,sg\}=\{g',s'g'\}$ and $\{h,th\}=\{h',t'h'\}$ holds by the definition of $\psi$, so $(g',s'g')=(g,sg)$ or $(sg,s)$, and $(h',t'h')=(h,th)$ or $(th,h)$. Suppose also that $e_{g',h',s',t'}\ne e_{g,h,s,t}$. Then it follows from part (a)(ii) that $e_{g',h',s',t'}= e_{g,th,s,t^{-1}}$ or $e_{sg,h,s^{-1},t}$, and by part (a) these are the same edge, but distinct from  $e_{g,h,s,t}$. Thus part (c) is proved.
	\end{proof}
	
	
	Now we discuss the Cayley codes. Our conclusions are not as strong as for Cartesian products.
	As input for a Cayley code construction for $\Delta$ in \eqref{e:cc}, we need a linear code in $\F^{S\times T}$, and it is most convenient to consider this space as the tensor product space $\F^S\otimes \F^T$. The space containing a Cayley code for $\Delta$ is $\F^{E\Delta}$ and the map $\psi$ induces a natural surjective linear homomorphism from $\F^{E\Delta}$ onto $\F^{E\Gamma}\otimes \F^{E\Sigma}$ (acting on bases for these spaces).
	So suppose that we are given linear codes $A \leq \F^S$ and $B\leq\F^T$, and that we have constructed $\C(G,S,A)\leq \F^{E\Gamma}$ and $\C(H,T,B)\leq \F^{E\Sigma}$. Then  $A\otimes B$ is a linear code in $\F^S\otimes\F^T$, and in Proposition~\ref{p:dprod} we identify a certain sub-code of the Cayley code $\C(G\times H, S\times T, A\otimes B)\leq \F^{E\Delta}$.
	Note that, by \eqref{e:cc}, for $\C:= \C(G\times H, S\times T, A\otimes B)$, we have
	\begin{align*}
		\C(G,S,A) &=\{ f^\Gamma\in\F^{E\Gamma} \mid \phi_g^\Gamma\circ f^\Gamma\in A\ \text{for all $g\in G$} \}, \\ 
		\C(H,T,B) &=\{ f^\Sigma\in\F^{E\Sigma} \mid \phi_h^\Sigma\circ f^\Sigma\in B\ \text{for all $h\in H$} \}, \ \text{and}\\ 
		\C &=\{ f^\Delta\in\F^{E\Delta} \mid \phi_{(g,h)}^\Delta\circ f^\Delta\in A\otimes B\ \text{for all $(g,h)\in G\times H$} \}.
	\end{align*}

	\begin{proposition}\label{p:dprod}
		With the notation as above, and $\psi$ as in Lemma~\ref{l:e},
		
		\begin{enumerate}
			\item[(a)] Let $(g,h)\in G\times H$ and let $\phi_g^\Gamma, \phi_h^\Sigma, \phi_{(g,h)}^\Delta$ be as in \eqref{e:phig} for the appropriate graphs. Then $ \phi_{(g,h)}^\Delta\circ \psi = \phi_g^\Gamma\otimes  \phi_h^\Sigma$. 
			
			\item[(b)] Let $f^\Gamma\in \C(G,S,A)$ and $f^\Sigma\in\C(H,T,B)$. Then $\psi\circ (f^\Gamma\otimes f^\Sigma)\in \C(G\times H, S\times T, A\otimes B)$. Thus $\C(G\times H, S\times T, A\otimes B)$ contains the preimage under $\psi$ of $\C(G,S,A)\otimes \C(H,T,B)$.
		\end{enumerate}
		Moreover, the second assertion of Theorem~$\ref{t:cprod}$ holds.
	\end{proposition}

	\begin{proof}
		Let $\C= \C(G\times H, S\times T, A\otimes B)$, as above.
		
		\medskip
		(a)  For $(s,t)\in S\times T$, the image $(s,t) (\phi_{(g,h)}^\Delta\circ \psi)$ is equal to
		\[
		\{ (g,h), (sg, th)\} \psi = \left( \{g,sg\},\ \{h,th\} \right) = (s\phi_g^\Gamma, t\phi_h^\Sigma) = (s,t) (\phi_g^\Gamma\circ\phi_h^\Sigma),
		\]
		proving part (a). 
		
		(b) Let $f=\psi\circ (f^\Gamma\otimes f^\Sigma)$. Then $f\in\F^{E\Delta}$. To show that $f\in\C$ we need to show that $\phi_{(g,h)}^\Delta\circ f\in A\otimes B$ for all $(g,h)\in G\times H$. By part (a), 
		$\phi_{(g,h)}^\Delta\circ f = (\phi_g^\Gamma\otimes  \phi_h^\Sigma) \circ (f^\Gamma\otimes f^\Sigma) = (\phi_g^\Gamma \circ f^\Gamma) \otimes  (\phi_h^\Sigma\circ f^\Sigma)$, which lies in $A\otimes B$ since $f^\Gamma\in \C(G,S,A)$ and $f^\Sigma\in\C(H,T,B)$. Since this holds for all $(g,h)$, we conclude that $f\in\C$. Thus part (b) is  proved. Finally we note that part (b) implies the second assertion of Theorem~\ref{t:cprod}.
	\end{proof}
	
	\subsection{Future directions}\label{s:future}    
	The decompositions of Cayley codes given in Theorem~\ref{thm:disccaycode} and Proposition~\ref{prop:SiBi} allow us, if desired, to focus on symmetric Cayley codes for connected Cayley graphs. On the other hand, we note that the most useful Cayley codes so far discovered have been those in \cite{KW2016} for a family of edge-transitive bipartite Ramanujan graphs for almost simple groups ${\rm PGL}_2(q)$. To understand the role that simple groups might play in the theory we ask:
	
	\begin{question}\label{q1}
		{\rm
			Is there a useful notion of a `normal quotient' of a (finite, symmetric) Cayley code which allows one to restrict to Cayley codes in the case where the group $A(\Gamma)$ in \eqref{e:Aut} for the Cayley graph $\Gamma$ is quasiprimitive or bi-quasiprimitive on vertices?
		}
	\end{question}
	
	The finite Cayley graphs $\Gamma$ for which $A(\Gamma)$ is quasiprimitive and bi-quasiprimitive on vertices are Cayley graphs for characteristically simple groups, that is,  groups of the form $T^k$ for a simple group $T$ and integer $k\geq1$ (\cite[Section 3, especially Theorem 4]{P99}). Recall that a Cayley graph $\Gamma$ is normal edge-transitive if $A(\Gamma)$ is edge-transitive. The open question \cite[Question 3]{P99} asks for a useful description of normal edge-transitive  Cayley graphs for finite characteristically simple groups. One might extend this question and ask:
	
	\begin{question}\label{q2}
		{\rm
			Is there a useful description of symmetric Cayley codes corresponding to normal edge-transitive Cayley graphs for finite characteristically simple groups?
		}
	\end{question}
	
	In Subsections~\ref{s:cartprod} and~\ref{s:dirprod} we analysed Cayley codes for Cartesian products and direct products of Cayley graphs. There are other common graph products, such as the strong product and the lexicographic product \cite[Section 7.15]{GR}. 
	
	\begin{question}\label{q3}
		{\rm
			How does the Cayley code construction behave under the strong product and the lexicographic product of  Cayley graphs?
		}
	\end{question}

	In Proposition~\ref{p:dprod}(b) we identified a subcode of the `direct product Cayley code'  
	$\C(G\times H, S\times T, A\otimes B)$ that projected to a tensor product of the Cayley codes for the factors. 
	
	\begin{question}\label{q4}
		{\rm
			Find families of examples of Cayley codes where $\C(G\times H, S\times T, A\otimes B)$ is equal to the subcode in Proposition~\ref{p:dprod}(b), and also families of examples where it is strictly larger. 
		}
	\end{question}
	
	We hope that such examples could give more insight into the structure of Cayley codes under the direct product construction.

	\subsection*{Acknowledgements and Declarations}
	The second author expresses her gratitude to Alex Lubotzky for bringing to our attention his work on Cayley codes, and the relevance of her work on normal edge-transitive Cayley graphs. 
	We are also grateful to the Centre for the Mathematics of Symmetry and Computation at the University of Western Australia for hospitality which allowed the three authors to participate in the 2023 CMSC Research Retreat where the work on this project started. 
	
	All authors contributed to the study conception and design.  Full drafts of the manuscript were written by Cheryl Praeger and Daniel Rademacher, and all authors commented on previous versions of the manuscript. All authors read and approved the final manuscript.
	
	The authors acknowledge funding for this research as follows: Australian Government Research Training Program (RTP) Scholarship, \texttt{ doi.org/10.82133/C42F-K220} (Arumugam);  Australian Research Coun\-cil Discovery Program Grant DP190100450 (Praeger); SFB-TRR 195 `Symbolic
	Tools in Mathematics and their Application' of the German Research Foundation
	(DFG), Program ID 286237555 (Rademacher).

	\medskip\noindent
	Vishnuram Arumugam  https://orcid.org/0000-0002-5229-2012\\ 
	Cheryl E. Praeger https://orcid.org/0000-0002-0881-7336\\
	Daniel Rademacher https://orcid.org/0009-0004-0638-0873

\end{document}